\documentclass[12pt]{amsart}
\usepackage{amsfonts}
\usepackage{amssymb}
\usepackage{tikz}
\usepackage{graphicx}
\usepackage[letterpaper, left=2.5cm, right=2.5cm, top=2.5cm,
bottom=2.5cm,dvips]{geometry}
\usepackage[normalem]{ulem}

\newtheorem{theorem}{Theorem}
\theoremstyle{plain}

\newtheorem{claim}[theorem]{Claim}

\newtheorem{conjecture}[theorem]{Conjecture}
\newtheorem{corollary}[theorem]{Corollary}

\newtheorem{lemma}[theorem]{Lemma}

\newtheorem{proposition}[theorem]{Proposition}

\numberwithin{equation}{section}

\begin{document}
\title[Localization game on geometric and planar graphs]{Localization game on geometric and planar graphs}
\author[B. Bosek]{Bart\l omiej Bosek}
\address{Theoretical Computer Science Department, Faculty of Mathematics and Computer Science, Jagiellonian
University, 30-348 Krak\'{o}w, Poland}
\email{bosek@tcs.uj.edu.pl}

\author[P. Gordinowicz] {Przemys\l aw Gordinowicz}
\address{Institute of Mathematics, Lodz University of Technology, \L{}\'od\'z, Poland}
\email{pgordin@p.lodz.pl}

\author[J. Grytczuk]{Jaros\l aw Grytczuk}
\address{Faculty of Mathematics and Information Science, Warsaw University
	of Technology, 00-662 Warsaw, Poland}
\email{j.grytczuk@mini.pw.edu.pl}

\author[N. Nisse]{Nicolas Nisse}
\address{Universit\'e C\^ote d'Azur, Inria, CNRS, I3S, France}
\email{nicolas.nisse@inria.fr}

\author[J. Sok\'{o}\l ]{Joanna Sok\'{o}\l }
\address{Faculty of Mathematics and Information Science, Warsaw University
	of Technology, 00-662 Warsaw, Poland}
\email{j.sokol@mini.pw.edu.pl}
\author[M. \'{S}leszy\'{n}ska-Nowak]{Ma\l gorzata \'{S}leszy\'{n}ska-Nowak}

\address{Faculty of Mathematics and Information Science, Warsaw University
	of Technology, 00-662 Warsaw, Poland}
\email{m.sleszynska@mini.pw.edu.pl}

\thanks{This work was supported by The National Center for Research (J. Grytczuk, J. Sok\'{o}\l, M. \'{S}leszy\'{n}ska-Nowak) and
	Development under the project PBS2/B3/24/2014, by the ANR project Stint (ANR-13-BS02-0007) (N. Nisse) and the associated Inria team AlDyNet (N. Nisse)}

\begin{abstract}
The main topic of this paper is motivated by a localization problem in
cellular networks. Given a graph $G$ we want to localize a
walking agent by checking his distance to as few vertices as possible. The
model we introduce is based on a pursuit graph game that resembles the famous Cops and Robbers game. It can be considered as a game theoretic variant of the \emph{metric dimension} of a
graph. We provide upper  bounds on the related graph invariant $\zeta (G)$, defined as the least number of cops needed to localize the robber on a graph $G$, for several classes of graphs (trees, bipartite graphs, etc). Our main result is that, surprisingly, there exists planar graphs of treewidth $2$ and unbounded $\zeta (G)$. On a positive side, we prove that $\zeta (G)$ is bounded by the pathwidth of $G$. We then show that the algorithmic problem of determining $\zeta (G)$ is NP-hard in graphs with diameter at most $2$. Finally, we show that at most one cop can approximate (arbitrary close) the location of the robber in the Euclidean plane.
\end{abstract}

\maketitle

\section{Introduction}

Among dozens of games studied in graph theory, the game of \emph{Cops and
	Robbers}~\cite{BonatoNowakowski} is one of the most intriguing. Two players, the cops and the
robber, occupy vertices of a graph $G$, and move along edges to neighboring
vertices or remain on their current vertex. The cops move first by occupying
a set of vertices. The robber then chooses a vertex to occupy, and the
players move in alternate rounds. The game is played with perfect
information, so the players see each others' moves. The cops win if they can
capture the robber by moving to a vertex the robber occupies; otherwise, the
robber wins. The \emph{cop number} of a graph $G$, denoted by $c(G)$, is the
minimum number of cops needed to win in $G$.

The game was introduced independently by Quilliot \cite{Quillot} and
Nowakowski and Winkler \cite{NowakowskiWinkler}. Since then it gained lots
of research activity resulting in many deep results, unexpected connections,
and challenging open problems. Also numerous variants motivated by
applications in diverse areas of mathematics and computer science have been
considered (see \cite{AignerFromm}, \cite{BonatoNowakowski}, \cite{Fomin}, 
\cite{Seymour}).

In this paper, we study another variant inspired by localization problems in
wireless networks. The recent growing popularity of mobile devices (iphones,
smartphones, etc.) stimulated lots of technological invention as well as
theoretical research for solutions of real-life localization tasks (see \cite%
{Bahl}, \cite{Gu}, \cite{Olsen}). In one of many possible approaches, a
network is modeled as a graph with radio signal receivers (such as Wi-Fi
access points) located at some vertices. The strength of the signal from a
mobile phone is proportional to its distance to particular receivers. So,
the information on the actual position of the mobile phone in the network is
only partial. Moreover, the mobile phone holder may walk along the network
changing his position in time. The task is to find him precisely at some
point by a procedure with prescribed efficiency.

Motivated by this scenario, we consider the following metric version of the
Cops and Robbers game. Let $G=(V,E)$ be a simple connected undirected graph and let $k\geq 1$ be a
fixed integer. 
The {\it localization game} involves two Players: the {\it Cop-player} (playing with a team of $k$ {\it cops}) and the {\it Robber-player} (the {\it robber}), and proceeds as follows. In the first turn, the robber chooses a vertex $r \in V$ unknown to the Cop-player (in the localization game, the robber is {\it a priori} ``invisible"). Then, at every turn, first the Cop-player picks (or probes) $k$ vertices $B=\{v_{1},v_{2},\ldots ,v_{k}\} \in V^k$ and, in return, gets the vector $D(B)=(d_{1},d_{2},\ldots ,d_{k})$ where $d_{i}=d_{G}(r,v_{i})$ is the distance (in $G$) from $r$ to $v_i$ for every $i=1,2,\ldots ,k$. If the Cop-player can determine the position of the Robber-player from $D(B)$ (i.e., $r$ is uniquely defined by $D(B)$ and the knowledge that the Cop-player has from previous turns), then the Cop-player wins. Note that $r$ is not required to be in $B$. Otherwise, the Robber-player may move to a neighboring node $r' \in N[r]$\footnote{In this paper, $N(v)$ denotes the set of neighbors of a vertex $v$ and $N[v]=N(v) \cup \{v\}$.} (unknown by the Cop-player) and then, during the next turn, the Cop-player can probe another set (possibly the same) of $k$ vertices, and if the robber is not located the robber may move to a neighbor, and so on. The cops win if they can precisely locate the position (the occupied vertex) of the robber after a finite number of turns (before the possible move of the robber). The robber wins otherwise. Let $\zeta (G)$ denote the
least integer $k$ for which the cops have a winning strategy whatever be the strategy of the robber (that is, we consider the worst case when the Robber-player {\it a priori} knows the whole strategy of the Cop-player). The graph-parameter $\zeta (G)$ is called the \emph{localization number} of a graph $G$. Notice that this
parameter is well defined since the inequality $\zeta (G)\leq \left\vert
V(G)\right\vert $ holds obviously. Note also that, by definition, $\zeta (G)$ is equal to the maximum  $\zeta(C)$ among the connected components $C$ of $G$. Therefore, from now on, only connected graphs are considered.

This game restricted to $k=1$ was introduced by Seager \cite{Seager1},
 and studied further in \cite{Brandt}, \cite{WestTCS}, \cite{Seager2}. It can be shown that, for every tree $T$, $\zeta(T) \leq 2$~\cite{Seager1}. Moreover, trees $T$ for which $\zeta(T)=2$ are  characterized in~\cite{Seager2}. More precisely, Seager~\cite{Seager1} proved that one cop is sufficient to localize a robber on any tree when robber is not allowed to move to a vertex just checked by the cop (in the previous round). In~\cite{Seager2}, she proved that this restriction is necessary for trees that contains a ternary regular tree of height $2$ as a subtree.

The localization game is also connected to the notion of \emph{metric dimension}\footnote{The metric dimension of a graph $G$ is the minimum cardinality of a subset $S$ of vertices such that all other vertices are uniquely determined by their distances to the vertices in $S$.} of a graph $G$%
, denoted by $\dim (G)$, introduced independently by Harary and Melter \cite%
{Harary}, and by Slater \cite{Slater}. Indeed, $\dim (G)$ can be defined as
the least number $k$ such that the Cop-player wins the localization game in one turn by cleverly
choosing $k$ vertices of $G$. Hence, the parameter $\zeta (G)$ can be seen
as the game theoretic variant of $\dim (G)$. 
The localization game has also been introduced recently in~\cite{Haslegrave}, where it is proven that $\zeta(G) \leq \lfloor \frac{(\Delta+1)^2}{4}\rfloor+1$ in any graph $G$ with maximum degree $\Delta$.

Our main goal in this paper is to find out for which classes of graphs the
number $\zeta (G)$ is bounded. Clearly, this holds for classes with bounded
metric dimension by the trivial inequality $\zeta (G)\leq \dim (G)$. We are
mainly interested in graph classes with bounded traditional cop
number $c(G)$. In particular we focus on graphs with various geometric representations because of their frequent applications to cellular networks, including  planar graphs for which $c(G)\leq 3$ as proved by Aigner and Fromme in \cite{AignerFromm}.

\smallskip

\noindent {\bf Our results.} As a warm-up, we give easy results and intuitions on the localization number (Section~\ref{sec:warmup}). 
Then, we show $\zeta (G) \leq k$ for any graph with pathwidth at most $k$ (Section~\ref{sec:pathwidth}). Our main result is that, somewhat surprisingly, $\zeta (G)$ is unbounded in graphs obtained from any tree by adding a universal vertex (i.e., planar graphs with treewidth at most $2$) (Section~\ref{sec:planar}). Then, we show that deciding whether $\zeta (G) \leq k$ is NP-complete in the class of graphs $G$ with diameter at most $2$ (Section~\ref{sec:complexity}). Finally, we prove boundedness of $\zeta (G)$ for some geometric graph in the plane (Section~\ref{sec:euclidean}).
In the final section (Section~\ref{sec:conclusion}) we set several open problems for future research.

\section{Warm-up}\label{sec:warmup}

Let us start with determining the localization number of some simple graphs. For instance, for a path $P_{n}$ on $n$ vertices, we have $\zeta (P_{n})=1$. Indeed, the Cop-player wins in one turn by starting from one of the ends of the path, which shows that also $\dim (P_{n})=1$. For complete graphs we have $\zeta(K_{n}) = \dim(K_n) = n-1$. On the other hand, for a star $S_{n}$ on $n$ vertices, $\dim (S_{n})=n-1$, while $\zeta (S_{n})=1$. To see this latter statement suppose that the Cop-player probes the leaves one by one. At some point the robber must be located. So, the difference between parameters $\zeta (G)$ and $\dim (G)$ can be arbitrarily large. 

Determination of $\zeta (G)$ for complete bipartite graphs is slightly less immediate. 

\begin{proposition}
	Every complete bipartite graph $K_{a,b}$ satisfies $\zeta (K_{a,b})=\min
	(a,b)$.
\end{proposition}

\begin{proof}
	Let $A$ and $B$ be the two partition classes of $K_{a,b}$ with size $a$ and $b$, respectively. We can assume that $a\leq b$. First we will show that $\zeta (K_{a,b})\leq \min (a,b)$. In the first round the Cop-player probes all vertices of $A$ except one,  and probes one vertex of $B$. If the robber occupies any vertex from $A$, the Cop-player wins, as she gets answer $0$ from one vertex from $A$ or all answers from vertices from $A$ are $2$ in which case the robber is in the unselected vertex from $A$. Otherwise, in next turns, the Cop-player changes only one vertex; she still probes all vertices of $A$ except one, and she probes a new vertex of $B$. If the robber stays all the time in the same of the vertices of $B$, the Cop-player will locate him at some round. If he moves to some vertex from $A$, she will also locate him.
	
	Now we will show that if the Cop-player can probe at most $\min (a,b)-1$ vertices in	each round, then the robber has a winning strategy. If the Cop-player probes all vertices from only one partition class, then the robber may choose any of vertices from the second class. If the Cop-player probes at least one vertex from $B$, then the robber may stay in one of the unselected vertices from $A$ (remind that the robber fully knows the strategy of Cop since we are considering a worst case). As there are at least two such vertices the robber remains hidden. The proof is complete.
\end{proof}

\begin{corollary}
	Every bipartite graph $G$ with partition classes of size $a$ and $b$
	satisfies $\zeta (G)\leq \min (a,b)$.
\end{corollary}

\begin{proof}
	The idea of the proof is the same as in case of complete bipartite graphs. This time, the cops recognize the partition class in which the robber is by parity of answers. If all answers from vertices from $A$ are even and different from $0$, then the robber is in the unselected vertex from $A$. Otherwise he occupies one of vertices from $B$ and the cops will eventually locate him in some future round.
\end{proof}

It is also worth noticing that in general parameter $\zeta (G)$ is not monotone on taking subgraphs. Let $G$ be as depicted in Figure~\ref{fig:mon}, and let $H=K_{4}$ be its subgraph. It is not hard to check that $\zeta(G)=2$ (by
starting with two added vertices), while $\zeta(H)=\zeta (K_{4})=3$.

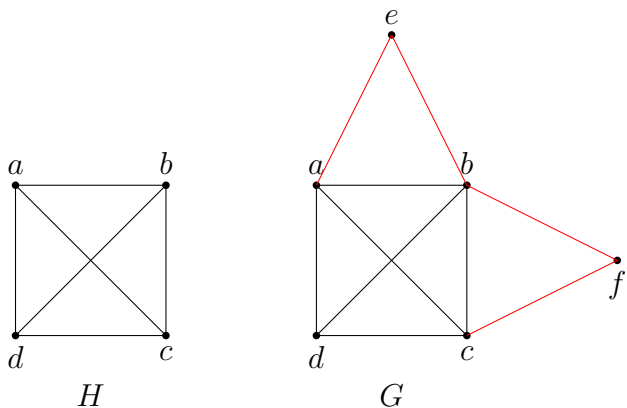
\begin{figure}[th]
	\centering
	\begin{tikzpicture}
	\coordinate (a) at (0, 2);
	\coordinate (b) at (2, 2);
	\coordinate (c) at (2, 0);
	\coordinate (d) at (0, 0);
	
	\node at (a) [above] {$a$};
	\node at (b) [above] {$b$};
	\node at (c) [below] {$c$};
	\node at (d) [below] {$d$};
	\node at (1, -0.5) [below] {$H$};
	
	\path[fill=black] (a) circle (0.05);
	\path[fill=black] (b) circle (0.05);
	\path[fill=black] (c) circle (0.05);
	\path[fill=black] (d) circle (0.05);
	
	\path[draw](a) -- (b);
	\path[draw](a) -- (c);
	\path[draw](a) -- (d);
	\path[draw](c) -- (b);
	\path[draw](d) -- (b);
	\path[draw](c) -- (d);
	
	\coordinate (a) at (4, 2);
	\coordinate (b) at (6, 2);
	\coordinate (c) at (6, 0);
	\coordinate (d) at (4, 0);
	\coordinate (e) at (5, 4);
	\coordinate (f) at (8, 1);
	
	\node at (a) [above] {$a$};
	\node at (b) [above] {$b$};
	\node at (c) [below] {$c$};
	\node at (d) [below] {$d$};
	\node at (e) [above] {$e$};
	\node at (f) [below] {$f$};
	\node at (5, -0.5) [below] {$G$};
	
	\path[fill=black] (a) circle (0.05);
	\path[fill=black] (b) circle (0.05);
	\path[fill=black] (c) circle (0.05);
	\path[fill=black] (d) circle (0.05);
	\path[fill=black] (e) circle (0.05);
	\path[fill=black] (f) circle (0.05);
	
	\path[draw](a) -- (b);
	\path[draw](a) -- (c);
	\path[draw](a) -- (d);
	\path[draw](c) -- (b);
	\path[draw](d) -- (b);
	\path[draw](c) -- (d);
	\path[color=red, draw](b) -- (f);
	\path[color=red, draw](c) -- (f);
	\path[color=red, draw](a) -- (e);
	\path[color=red, draw](b) -- (e);
	
	\end{tikzpicture}
	\caption{$\zeta(H) = 3$, $\zeta(G) = 2$}
	\label{fig:mon}
\end{figure}

\section{Pathwidth}\label{sec:pathwidth}

In this section, we show that the localization number of a graph is bounded from above by the pathwidth of the graph.

A {\it path-decompositon} of a graph $G=(V,E)$ is a sequence ${\mathcal X}=(X_1,\cdots,X_t)$ of subsets of $V$, called {\it bags}, such that, for every edge $\{u,v\} \in E$, there exists a	bag containing both $u$ and $v$, and such that, for every $1\leq i \leq k \leq j \leq t$, $X_i \cap X_j \subseteq X_k$. The {\it width} of $\mathcal X$ equals $\max_{1\leq i \leq t} |X_i|-1$ and the {\it pathwidth}  of $G$, denoted by $pw(G)$, is the minimum width of its path-decompositions. Pathwidth and path-decompositions are closely related to some kind of pursuit-evasion games~\cite{Bienstock}.

\begin{proposition}
Every graph $G$ satisfies $\zeta(G) \leq pw(G)$. Moreover, this bound is achieved for interval graphs. 
\end{proposition}

\begin{proof}
We may assume that $G$ is connected and has at least two vertices.	Let $(X_1,\cdots,X_t)$ be an optimal path-decomposi\-tion (of width $pw(G)$) of $G$. We may assume that $|X_{i} \setminus X_{i+1}|\geq 1$  for every $1\leq i< t$ (otherwise, if $X_{i} \subseteq X_{i+1}$, we may consider the path decomposition  $(X_1,\cdots,X_{i-1},X_{i+1},\cdots,X_t)$). We also may assume that, for every $1\leq i< t$ and every $u \in X_{i} \setminus X_{i+1}$, $u$ has a neighbor in $X_{i}$ (otherwise, because $G$ is connected, we may consider the path decomposition  $(X_1,\cdots,X_{i-1},X_{i} \setminus \{u\}, X_{i+1},\cdots,X_t)$).

For every $1\leq i < t$, let $u_{i}$ be any vertex in $X_{i} \setminus X_{i+1}$ and let $v_{i}$ be a neighbor of $u_{i}$ in $X_{i}$. Finally, let $v_t$ be any vertex in $X_t \setminus X_{t-1}$ and $u_t$ be any neighbor of $v_t$ ($u_t$ exists since $G$ is connected and belongs to $X_t$). Sequentially, for $i=1$ to $t$, the Cop-player probes $X_i \setminus v_i$. We will show by induction on $i$, that before the $i^{th}$ turn of the Cop-player, the robber must occupy some vertex in $\bigcup_{i\leq j} X_j$ (this  clearly holds for $i=1$). If the robber is in $X_i \setminus v_i$, he is located. Moreover, if he is in $v_i$, he is located because the robber is at distance $1$ from $u_i$ (and not already localized) iff he is at $v_i$. 	
	If not located, the robber must be in $G \setminus (\bigcup_{j\leq i} X_j)$. By moving, the robber can only reach some vertex in $\bigcup_{i+1\leq j} X_j$. Hence, the induction hypothesis holds. Eventually (for $i=t$), the robber will be located.
	
	To show that the bound is achieved in interval graphs, consider the strategy where the robber stays in a maximum clique of an interval graph $G$ (i.e., a clique of size $pw(G)+1$). Using at most $pw(G)-1$ probes, there are always two vertices of the clique that are not distinguishable and where the robber may be. 
\end{proof}

\section{Planar graphs; blind cops and bush cutting}\label{sec:planar}

This section is devoted to prove that the localization number of planar graphs (even for graphs obtained from a tree by adding to it a universal vertex, in particular, with treewidth $2$) is unbounded. 
To this end we introduce another variant of the localization game in which the cops are \emph{blind}. The game proceeds as follows. First the robber chooses a vertex of a graph $G$ (unknown to the cops). Then $k$ blind cops probe $k$ vertices of $G$ and their neighbors. That is, if the robber is occupying one of the probed vertices or any neighbor of the probed vertices, it is immediately caught. Otherwise, the cops only learn that the robber do not occupy any of these vertices. If the robber is caught the game is over. Otherwise, in next round robber may walk to a neighboring vertex while cops may choose new vertices arbitrarily. Let $\zeta_b(G)$ denote the least number of blind cops needed to catch the robber on a graph $G$. Note that $\zeta_b(G)$ and $\zeta(G)$ are {\it a priori} not comparable since, in the blind game, the cops have less power if the robber is ``far" (they do no get any information about its distance) but the cops are more powerful if the robber is ``close", since the robber is precisely located (and so, caught) at any vertex of the closed neighborhood of the probed vertices. 

However, the following proposition shows relation between the two localization games.

\begin{proposition}\label{prop:blind}
For a given graph $G$, let $G'$ be a copy of $G$ with one additional vertex $v$ adjacent to all vertices of $G$. Then $\zeta_b(G)\leq \zeta(G').$
\end{proposition}
 
\begin{proof}
Suppose that $\zeta(G')=k$, and consider localization game on $G'$ with $k+1$ cops, with one extra cop placed on the new vertex $v$ during the whole game. This forces robber to use only vertices of $G$, but of course, the cops still have a winning strategy. Moreover, each situation in the usual localization game on $G'$ is now transformed into the blind cops game on $G$ (removing the cop in $v$). That is so, because the distance in $G'$ from the robber to any vertex occupied by a cop is either one or two, which corresponds to the situation when the robber is in the neighborhood of a cop or not (in the blind game on $G$). By the assumption the usual game ends with localizing the robber at some vertex $x$. Hence, the blind game can be ended in the next move by placing a cop on that vertex $x$.
\end{proof}

Our aim is now to prove that there exists a sequence of trees $\mathcal{T}_k$ such that $\zeta_b(\mathcal{T}_k)>k$. By Proposition~\ref{prop:blind}, this will prove unboundedness of $\zeta(G)$ for planar graphs since each graph $\mathcal{T'}_k$ is planar. To achieve this goal we introduce a slightly modified game just to simplify the logic of forthcoming arguments.

Consider the following turn-by-turn game. 
Let $k$ be a positive integer.
Initially, a {\it bush} is present on each vertex of a graph and the goal of $k$ {\it bush-cutters} is to remove it completely. 
At every turn, every bush-cutter chooses a vertex and removes the bush from the closed neighborhood of it. Then, the bush extends to every neighbor of the vertices where the bush is still present.
Let the \emph{bush number} of $G$, denoted by $B(G)$, be the least number of bush-cutters for which there is a strategy to remove the bush. Clearly by choosing any dominating set bush-cutters can cut the whole bush in one round. Hence the bush number is well defined and $B(G) \le \gamma(G)$, where $\gamma(G)$ denotes the size of a smallest dominating set of $G$. 

Note that this model is connected with the blind cops game strategy in the following sense.

\begin{proposition}\label{prop:bush}
Every graph $G$ satisfies $B(G) \leq \zeta_b(G)$.
\end{proposition}

\begin{proof}
Suppose for a contrary that $k$-cutters never can cut the whole bush on $G$, i.e., $k<B(G)$. We show that  $k<\zeta_b(G)$, that is, we show that the robber can win against $k$ blind cops. It means that we aim at describing an infinite strategy of the robber that prevents the $k$ blind cops to localized it. Since the strategy of the blind cops is deterministic and the number of configurations (positions of the cops and of the robber) is finite, it is sufficient to show that there is a long enough strategy of the robber that prevents the $k$ blind cops to localized it. We actually show that, for every $t\in \mathbb{N}$ and for every strategy of the cops, there is a strategy of the robber that prevents the $k$ blind cops to localized it for at least $t$ rounds. 

Let us consider any strategy $\mathcal S$ of $k$ blind cops (i.e., at every turn, $k$ vertices are probed). Note that, in the blind game, the cops get no information unless they localize the robber and win. Therefore, any strategy is actually defined as an infinite sequence $(v^i_1,\cdots,v^i_k)_{i \in \mathbb{N}}$, where the vertices in $(v^i_1,\cdots,v^i_k)$ are the ones that are probed at the ith turn. The strategy is winning if there is a turn $i$ when the robber is occupying a vertex in the closed neighborhood of the probed vertices at this turn.

Let us show that, for every $t \in \mathbb{N}$, there is a strategy for the robber to win against $\mathcal S$ during at least $k$ turns. Let ${\mathcal S}'$ be the strategy of $k$ brush-cutters that follows $\mathcal S$ (i.e., when one vertex $v$ is probed in $\mathcal S$, the bush is cut in $N[v]$ in ${\mathcal S}'$). Since $k<B(G)$, after $t$ turns, there is a vertex $x \in V(G)$ that still belongs to the bush. By definition of the growth of the bush, there exists a sequence $W=(x_1,\cdots,x_t)$ of vertices such that, for every $1\leq i\leq t$, $x_{i+1} \in N[x_{i}]$ and $x_i$ is in the bush after the ith turn of the cut-bushers.

Note that for every turn $i$, there is $x_i \notin \bigcup_{j\leq k} N[v^i_j]$. Hence, the robber just has to follow the walk $W$ in order to win $t$ turns against $\mathcal S$.
\end{proof}

Before proving the main result of this section we will need the following technical lemma on bicolored matchings in trees.
Let $(T,f)$ be any tree together with a coloring $f \colon V(T) \to \{0, 1\}$ where $f$ is any {\it $2$-vertex coloring} of the tree $T$ (not necessary proper). Recall that a {\it matching} is a set of disjoint edges. We say that a matching $M \subseteq E(T)$ of a $2$-colored tree $(T,f)$ is {\it bicolored} if every edge $uv \in M$ has endpoints in different colors, in other words $f(u) \neq f(v)$. We say that $(T,f)$ is {\it monochromatic} if $f^{-1}(1)=V(T)$ or $f^{-1}(0)=V(T)$, i.e., if all vertices receive the same color. The following claim is obvious:
\begin{claim}\label{claim}
If $(T,f)$ is not monochromatic, then there is a bicolored edge.
\end{claim}

Given a tree $T$ rooted in $v \in V(T)$, an {\it r-subtree} $T'$ of $T$ is any sub-tree of $T$ rooted in a child of the root vertex $v$. For any $k\geq 1$ and $h \geq 1$, let $T^k_{h}$ be the complete $(12k+1)$-ary rooted tree of height $h$. Let $n_{h,k}=|V(T^k_{h})|=(12k+1)^h+1$.

\begin{lemma} \label{lem:bimatching}
	Let $k\geq 1$, $T^k_{h}$ be complete $(12k+1)$-ary rooted tree of height $1\leq h \le 6k$.  Let $f \colon V(T^k_{h}) \to \{0, 1\}$ be any $2$-vertex coloring of $T^k_{h}$ such that
\begin{equation}\label{eq:bimatching}
\frac{n_{h,k}+h-8k}{2} \le |f^{-1}(\{1\})| < \frac{n_{h,k}+6k-h}{2}.
\end{equation} 
Then there exists a bicolored matching of size at least $h$.
\end{lemma}

\begin{proof}
	The proof is by induction on $h\geq 1$. For base step, $h = 1$ note that a tree $T^k_1$ have between $2k$ and $9k$ leaves with color 1. Hence, since $T^k_1$ has $12k+1$ leaves, $T^k_1$ is non monochromatic, there is a bicolored edge (matching of size 1). 

	Now, let $h$ be such that $1 < h \le 6k$, and let us assume by induction that the lemma is true for every $h' < h$. Let $f$ be any $2$-coloring of $T^k_h$ satisfying (\ref{eq:bimatching}), and let $M$ be a corresponding maximum bicolored matching. We show that $|M| \ge h$. 
	
	
	If $(T^k_h,f)$ contains at least $6k$ (vertex-disjoint) non-monochromatic r-subtrees, then there is a bicolored matching $M$ with $|M|\geq 6k\geq h $ by Claim~\ref{claim} and the lemma is proved. 	
	Hence, we may assume that $T^k_h$ contains at most $6k-1$ non-monochromatic r-subtrees, i.e., at least $6k+2$ monochromatic $r$-subtrees.
	
	By (\ref{eq:bimatching}), for $y \in \{0, 1\}$, the number of monochromatic $y$-colored r-subtrees is not greater than $6k$. Therefore, there must be at least one monochromatic r-subtree colored with $0$ and at least one colored with $1$. Hence, the root of $T^k_h$ is an endpoint of some edge of $M$.  There are two cases to be considered.
	\begin{enumerate}
	\item If a unique r-subtree (say $T'$) is not colored monochromatically by $f$, then we are done. Indeed, in this case, by above paragraph, there are exactly $6k$ monochromatic $r$-subtrees colored with $0$ and $6k$ monochromatic $r$-subtrees colored with $1$. Note that the number of vertices $n_{h-1,k}$ of any $r$-subtree satisfies $(12k+1)n_{h-1,k} + 1= n_{h,k}$. Hence, by (\ref{eq:bimatching}), the number $|f^{-1}_{|T'}(\{1\})|$ of vertices colored with $1$ in $T'$ is between $\frac{n_{h,k}+h-8k}{2}-6k*n_{h-1,k}-1$ (the last "minus one" is in case the root of $T^k_h$ is colored with $1$) and $\frac{n_{h,k}+6k-h}{2}-6k*n_{h-1,k}$. That is, 
	
	$$ \frac{n_{h-1,k}+(h-1)-8k}{2} \leq |f^{-1}_{|T'}(\{1\})| < \frac{n_{h-1,k}+6k-(h-1)}{2}$$
	
Therefore, by the inductive assumption, $T'$ contains a matching of at least $h-1$ bicolored edges. Adding the bicolored edge between one monochromatic $r$-subtree and the root of $T^k_h$, we get that a bicolored matching of size at least $h$ in $T^k_h$.
	
	\item Suppose now that there are $x \ge 2$ non monochromatic r-subtrees. Enumerate all r-subtrees by $T_1, T_2, \dots, T_{12k+1-x}, T_{12k+2-x}, \dots, T_{12k+1}$, in such a way that, for every $i \le 12k+1-x$, the r-subtree $T_i$ is monochrome (while the $r$-subtrees $T_{12k+2-x}, \dots, T_{12k+1}$ are not monochomatic). Consider now the subtrees of $T^k_h$ rooted in the grand-children of the root of $T^k_h$, i.e., the r-subtrees of the $T_i$'s. To avoid confusion, let us refer to these $(12k+1)(12k+1)$ subtrees as the s-subtrees.
	
	Note that there are less than $6k$ non monochromatic such s-subtrees since otherwise, by Claim~\ref{claim}, we would already get a bicolored matching of size at least $6k$. Combined with (\ref{eq:bimatching}), this implies that, for $y\in \{0,1\}$, there are at least $(12k+1)6k+1$ monochromatic $s$-subtrees colored with $y$. Therefore, at least one of $T_i$'s has monochromatic s-subtree in each color. 
	Renumbering when necessary, assume that it is $T' = T_{12k+1}$. We claim that it is possible to improve the coloring $f$ to a new coloring $f'$, satisfying $|f'^{-1}(\{1\})| = |f^{-1}(\{1\})|$  such that the maximum bicolored matching $M'$ corresponding to $f'$ is such that $|M'| \le |M|$ and all trees $T_i$ for $i < 12k+1$ are monochromatic. Therefore, we are back to the case (1) when there is exactly one non-monochromatic $r$-subtree which finishes the proof.   
	
	Let $z < 6k$ be the number of non monochromatic s-subtrees. Note that, the $x$ non monochromatic r-subtrees $T_i$ ($i \ge 12k + 2 - x$) contain $(12k+1)x - z$ monochromatic s-subtrees. Let $t_y$ be the number of monochromatic s-subtrees with color $y\in \{0,1\}$. Note that, according to (\ref{eq:bimatching}) one has $$t_0 \!\!\mod (12k+1) + t_1 \!\!\mod (12k+1) + z = 12k+1.$$ 
Indeed, for $h = 2$ there is $z = 0$. For $h > 2$ consider such a coloring of these $z$ s-subtrees that numbers of $y$-colored vertices are the same and at most one s-subtree is non monochrome, clearly there is $z_y < z$ $y$-colored s-subtrees for $y \in \{0, 1\}$. Notice that (from (\ref{eq:bimatching})) there is $t_y \mod (12k+1) + z_y = 6k$, while $z_y < z < 6k$.
	
	We want to exchange these s-subtrees (precisely, we aim at exchanging their coloring) between r-subtrees to minimize the number of r-subtrees containing s-subtrees in both colors. So, for $y \in \{0, 1\}$, let $x_y = \lfloor \frac{t_y}{12k+1} \rfloor$ be the number of desired monochromatic trees $T_i$ with $i \ge 12k+1-x$. Clearly $x_0 + x_1 = x-1$, moreover at least $x_y$ r-subtrees $T_i$ contains monochromatic s-subtree in a color $y$. For $i$ between $12k+2-x$ and $12k$ assign $g_i$ as the desired color of the r-subtree $T_i$ following the two rules: the tree $T_i$ contains monochromatic s-subtree in color $g_i$ and among values of $g_i$, there are precisely $x_0$ zeros and $x_1$ ones. 
	
	By a simple counting argument, for at least one color $y \in \{0, 1\}$ there are at least two monochromatic $y$-colored s-subtrees in $T'= T_{12k+1}$ and there is a tree $T_i$ with $g_i = y$ such that $T_i$ contains an s-subtree which is not in color $y$ (either non monochrome or in other color). We improve the coloring exchanging those two subtrees. Note, that it does not increase the size of the bicolored matching (as $T'$ contains another monochromatic subtrees in both colors, while $T_i$ in color $y$).
We repeat this procedure until each $r$-subtree $T_i$, $i < 12k+1$, contains only monochromatic s-subtrees in one color. Now, for each $r$-subtrees $T_i$ for which the root has a color different than $g_i$, let us exchange the color of this root with the color of some vertex from the tree $T'$ with the opposite color. Clearly it removes one edge from a matching (in the tree $T_i$) and adds at most one in $T'$. Therefore it does not increase the total size of the bicolored matching. Finally, we have obtained a coloring $f'$ as claimed which finishes the proof.  	
	\end{enumerate}
\end{proof}

We are now ready to prove the aforementioned result.

\begin{theorem}\label{theo:bush}
	For every $k\geq 1$, there exists a tree $\mathcal{T}_k$ such that $B(\mathcal{T}_k) > k$.
\end{theorem}

\begin{proof}
	Let $k, i \in \mathbb N$. Let $T_i^k$ be the complete $(12k+1)$-ary rooted tree of height $i$. The number of vertices of $T_i^k$ satisfies 
	$$|V(T_i^k)| = \frac{(12k+1)^{i+1}-1}{12k}.$$
	The tree $\mathcal{T}_k$ is constructed by subdividing twice each edge of $T_{6k}^k$. The vertices of $\mathcal{T}_k$ of degree different than 2 are said \emph{regular} vertices, while vertices of degree 2 --- \emph{subdivision} vertices.
	For regular vertices we will use a terminology of parents, children, ancestors and descendants (regarding corresponding $T_{6k}^k$ tree), while for the whole tree --- a terminology of neighbors.
	
Let $k\geq 1$. Let $n = |V(T_{6k}^k)|$ denote the number of regular vertices in $\mathcal{T}_k$. We show that $k$ bush-cutters cannot clean a tree $\mathcal{T}_k$. Suppose otherwise. Then there exists a strategy of cutting allowing to clean the tree. For each time step $t$ (before cutters move) denote by $m_t$ number of clean (without bush) regular vertices. 

Let now $t$ be the last round when, before cutters move, bush grows on more than a half of regular vertices, so this is the last moment when $m_t < \frac{n}{2}$. As at each time step cutters can cut at most $k$ regular vertices the number $m_t$ of clean regular vertices satisfies inequality $\frac{n}{2}-k \le m_t$.

By Lemma~\ref{lem:bimatching} (color vertices of $T_{6k}^k$ by 0 or by 1 whether corresponding regular vertices of $\mathcal{T}_k$ are bush or clean respectively), there exists a bicolored matching $M$ of size $6k$ in $T_{6k}^k$  (whose edges are between bush and clean vertices). In the tree $\mathcal{T}_k$, the edges of $M$ correspond to paths (say $M$-paths) of length 3 connecting bush and clean vertices with the property that each vertex belong to at most one such path.

In three consecutive rounds cutters can cut the bush on at most $3k$ endpoints of $M$-paths. While on at least $3k$ such an endpoints bush regrows. Hence, $m_{t+3} \le m_t < \frac{n}{2}$. A contradiction.
\end{proof}

From Propositions~\ref{prop:blind} and~\ref{prop:bush} and Theorem~\ref{theo:bush}, we get that:

\begin{corollary}
For any $k >0$, there exists a planar graph $G$ with treewidth $2$ (precisely, a tree plus a universal vertex) such that $\zeta(G)>k$. 
\end{corollary}

\section{Complexity}\label{sec:complexity}

In this section, we prove that the localization game 
(i.e., computing $\zeta$)
is NP-hard.


We first introduce some related problems. 
A set $L \subseteq V$ of vertices is called a {\it locating set} if, for every $u,v \in V \setminus L$, $N[u] \cap L \neq N[v] \cap L$. 
Note that a locating set must ``see" almost all vertices. Formally, for any locating set $L$, $|V \setminus N[L]| \leq 1$. Indeed, otherwise, there would be two vertices $u,v$ such that $N[u] \cap L = N[v] \cap L =\emptyset$.
A set $LD \subseteq V$ of vertices is called a {\it dominating-locating set} if $LD$ is a dominating set and, for every $u,v \in V \setminus LD$, $N[u] \cap LD \neq N[v] \cap LD$. 
By definition,  the minimum size of a  dominating-locating set is at least the minimum size of locating set. Moreover, by remark above the minimum size of a  dominating-locating set is at most one plus the minimum size of locating set. It is known that:


\begin{theorem}~\cite{Cohen,Charon}
	Computing a minimum dominating-locating set is NP-hard.
\end{theorem}

We first deduce the following easy result

\begin{corollary}
	Computing a minimum locating set is NP-hard in the class of graphs with diameter $2$.
\end{corollary}
\begin{proof}
	We first show that computing a minimum locating set is NP-hard in general graphs. We present a reduction from the problem of computing a minimum dominating-locating set which is NP-hard by previous Theorem. Let $G$ be any graph and let $G'$ be the graph obtained from $G$ by adding to it an isolated vertex $x$. Then, $G$ has a locating-dominating set of size $k$ if and only if $G'$ has a locating set of size $k$. Indeed, 
	\begin{itemize}
		\item
		assume that $C \subseteq V(G)$ is a locating-dominating set of $G$, then it is clearly a locating set of $G'$ (where $x$ is the only vertex with $N[x] \cap C = \emptyset$). 
		\item
		On the other hand, let $C \subseteq V(G')$ be locating set of $G'$. If $C \cap V(G)$ is a dominating set of $V(G)$, then $C$ is a locating-dominating set of $G$. Otherwise, a unique vertex $v \in V(G)$ is not dominated by $C$. In that case, $x$ must belong to $C$ (otherwise we would have $N[x] \cap C= N[v] \cap C = \emptyset$). Therefore, $C \cup \{v\} \setminus \{x\}$ is a locating-dominating set of $G$ with same size than $C$.
	\end{itemize} 
	
	Now, let us show that computing a minimum locating set is NP-hard in the class of graphs with a universal vertex. The reduction is from the locating set problem which is NP-hard by previous paragraph. Let $G$ be any graph with at least two vertices, and let $G'$ be the graph obtained from $G$ by adding three new vertices $u,v,w$, the edge $\{v,w\}$ and making $u$ universal, i.e., adjacent to every vertex of $V(G) \cup \{v,w\}$. Then, $G$ has a locating set of size $k$ if and only if $G'$ has a locating set of size $k+1$. Indeed, 
	\begin{itemize}
		\item
		if $C \subseteq V(G)$ is a locating set of $G$, then $C \cup \{v\}$ is a locating set for $G'$. 
		\item
		Now, let $C \subseteq V(G')$ be a minimum locating set of $G'$. It is easy to check that exactly one of $v$ or $w$ must belong to $C$ (if none of them belongs to it, they cannot be distinguished, and if both of them belongs to $C$, one of them can be removed) and that $u$ does not (adding $u$ in $C$ would differentiate only itself while it is already identified by $v$ or $w$). W.l.o.g., we may assume that $C \cap \{u,v,w\} = \{v\}$. Therefore, it is also easy to check that $C \setminus \{v\}$ is a locating set for $G$.
	\end{itemize}
\end{proof}

We are now ready for our main result

\begin{theorem}
	The localization game 
	is NP-hard.
\end{theorem}
\begin{proof}
	We present a reduction from the problem of computing a minimum locating set in graphs with diameter $2$ which is NP-hard by previous corollary. 
	
	Let $G$ be any $n$-node graph with diameter $2$. Let $G'$ be the graph obtained from $G$ by adding $n+1$ pairwise non-adjacent vertices $x_1,\cdots,x_{n+1}$, each of them being adjacent to every vertex of $V(G)$. Note that, because $G$ has diameter $2$, then $G$ is an isometric subgraph of $G'$ (i.e., distances are preserved). Let $k$ be the minimum size of a locating set of $G$. We prove that $\zeta(G')=k+1$. 
	\begin{itemize}
		\item First, let us prove that $\zeta(G') \leq k+1$. Let $C$ be a locating set of $G$ with size $k$. The game will last at most $n+1$ turns. At turn $i$, the Cop-player probes the vertices in $C \cup \{x_i\}$. We show that the robber can never reach a vertex of $V(G)$ without being immediately caught. Therefore, the robber is stacked at its initial position in $\{x_1,\cdots,x_{n+1}\}$, say $x_j$. Hence, it will be caught at turn $j$. 
		
		First, at every turn $i$, the Cop-player can decide if the location $r$ of the robber belongs to $V(G)$ or not. Indeed, the robber is located at $\{x_1,\cdots,x_{n+1}\}$ if and only if the distance between $r$ and $x_i$ is strictly larger than the distance between $r$ and every other vertex of $C$. 
		
		Now assume that $r \in V(G)$ just before any turn $i$ and that $r \notin C$ (otherwise it is immediately caught). For every vertex $v \in C$, $r \in N(v)$ if and only if the distance between $v$ and $r$ equals the distance between $r$ and $x_i$. Since $C$ is a locating set (and recall that $G$ is isometric in $G'$), $r$ can be identified. 
		\item Finally, let us prove that $\zeta(G') \geq k+1$. For this purpose, let us describe an escape strategy for the robber when the Cop-player can probe at most $k$ vertices at each turn. Note that, since (obviously) $k<n$, at every turn, at least two vertices of $\{x_1,\cdots,x_{n+1}\}$ are not probed.
		\begin{itemize}
			\item If the robber is currently occupying a vertex of $V(G)$, then it goes to a vertex $x_i$ that will not be probed by the Cop-player (recall that we assume the worst case, i.e., the robber knows the whole strategy of the Cop-player). By the remark above (two vertices of $\{x_1,\cdots,x_{n+1}\}$ are not probed), the location of the robber cannot be identified.
			\item If the robber currently occupies a vertex in $\{x_1,\cdots,x_{n+1}\}$ and the Cop-player is about to probe $k$ vertices in $V(G)$, then the robber stays idle.
			\item Finally, consider the case when the robber occupies a vertex  in $\{x_1,\cdots,x_{n+1}\}$ and the Cop-player will probe $t \leq k-1$ vertices in $V(G)$ (and $k-t$ vertices in $\{x_1,\cdots,x_{n+1}\}$). Let $Y \subseteq V(G)$ be the set of vertices of $V(G') \setminus \{x_1,\cdots,x_{n+1}\}$ that will be probed. 
			
			Note that because the graph $G'$ has diameter $2$, for any probed vertex $v$, the distance between $v$ and the location $r$ of the robber only gives the information of whether $r=v$ or $r \in N(v)$ or $r \notin N[v]$ (similar information as given by a locating set). 
			
			Since $Y$ is not a locating set (because $|Y|<k$), at least two vertices $u$ and $v$ cannot be distinguished by $Y$ in $G$. Then, the robber goes to one of these vertices (note that the probed vertices in $\{x_1,\cdots,x_{n+1}\}$ bring no more information allowing to distinguish $u$ and $v$). 
		\end{itemize}
	\end{itemize}
\end{proof}

\section{Geometric localization game}\label{sec:euclidean}

Consider an infinite graph $G_{1}$ whose vertices are all points of the
plane with edges between points at Euclidean distance at most one. In other
words, $G_{1}$ is the intersection graph of all unit disks in the plane.
Notice that the graph distance between two vertices in $G_{1}$ is the least
integer not smaller than the Euclidean distance between the corresponding
points. In the first result we show that a countable number of cops is
not sufficient to localize the robber exactly on $G_{1}$. Recall that $\aleph _{0}$ denotes the cardinality of the natural numbers.

\begin{proposition}
$\zeta (G_{1})>\aleph _{0}$.
\end{proposition}

\begin{proof}
We will show that after each round there are uncountably many possible locations for the Robber-player.
Let $C$ be the set of probed points. For any $c\in C$ consider a set of all circles with center $c$ and integer radius. Such circles divide the plane into regions according to their distance from $c$. 
Let us denote the set of possible positions of the Robber-player by $R$. Assume that after the previous rounds of the game there are uncountably many points in $R$ (which is true at the beginning of the game). Consider the partition of $R$ created by all such circles around all points from $C$.
Since there are countably many circles, at least one region of the partition must contain uncountably many points. Hence if Robber chooses to stay in that region he will not be localized.
\end{proof}

In view of the above result we consider a relaxed version of the game in
which Cop measures the actual Euclidean distance to robber's location. We
will call it the\emph{\ geometric localization game} on the plane.

\begin{theorem}
The following statements hold for the geometric localization game on the
plane:

\begin{enumerate}
\item Three cops can win in one round.

\item Two cops can win in two rounds.

\item One cop cannot win in any number of rounds.
\end{enumerate}
\end{theorem}

\begin{figure}[htb]
\center
\includegraphics[width=0.7\textwidth]{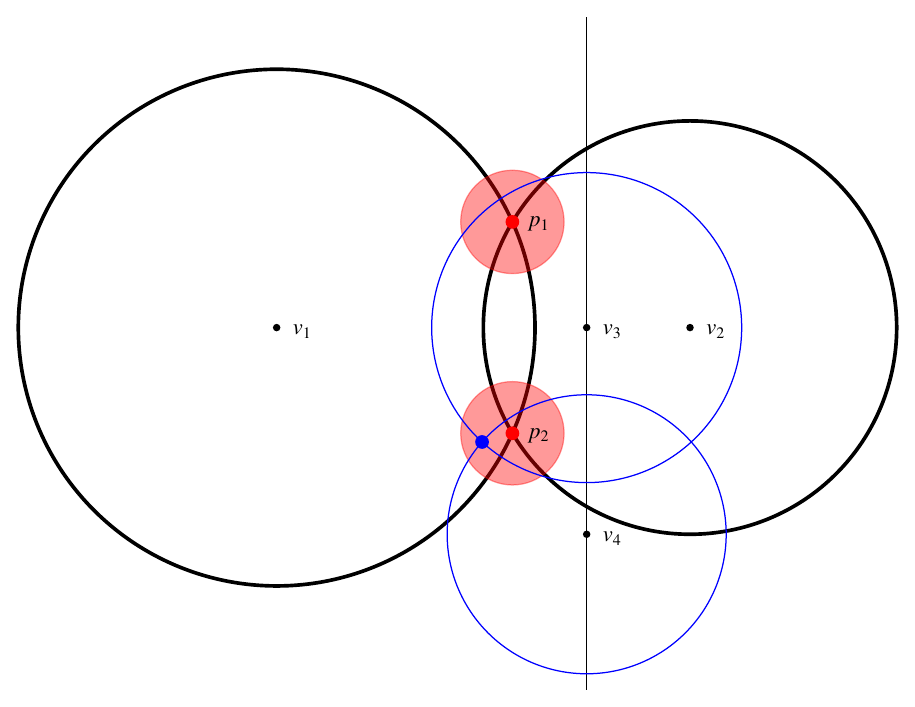}
\caption{The winning strategy in two moves for the Cop }\label{cop2}
\end{figure}

\begin{proof}
For the first statement, the winning strategy for the Cop-player is to choose three
noncolinear points. Then for every possible location of the robber, the
sequence of distances is unique, so he is immediately localized.

For the second statement, the winning strategy (presented in Figure \ref%
{cop2}) for the Cop-player is the following:
\begin{itemize}
\item
In the first round, the Cop-player checks any two points $v_{1}$ and $v_{2}$. If the robber is not on one of these points (which would end the game), then
each of $v_{1}$ and $v_{2}$ gives a circle on which the robber is located. If these cycles
intersect in only one point, this is the location of the robber and the Cop-player wins.
Otherwise, the circles intersect in two points: $p_{1}$ and $p_{2}$, and the 
Cop-player needs the second round to win.
\item
In second round, the Cop-player already knows that the robber was in one of the two
possible locations $p_{1}$, $p_{2}$, and he moved by at most one, hence he
is in one of the two unit disks with centers in $p_{1}$ and $p_{2}$ (the
disks may have nonempty intersection). Now, the Cop-player can draw a line such that
both disks lie entirely on the same side of the line. Hence the robber is on
that side of the line. Then, the Cop-player  chooses any two points on the line $v_{3}$
and $v_{4}$ and asks for their distances to the robber. Again she gets two
circles on which the robber is located, which intersect in at most two
points. But now only one of those points is on the same side of the line as
the disks, and that point is the location of the robber.
\end{itemize}

To prove the third assertion of the theorem we will show that the robber has a
winning strategy even if, at the beginning of each round, he tells the Cop-player  his
previous location $p$. That implies that, at the beginning of each round, the Cop-player
knows that the current location of the robber is in a unit disk $D$ centered in $%
p$. Then, the Cop-player chooses a vertex $c$ to check and the robber makes his move. Note
that, as long as the robber avoids point $c$, the distance $d$ between $c$ and
the robber is greater than zero. This gives a circular arc (intersection of a
circle with $D$) with center in $c$ and radius $d$, where the robber must be
located. If the robber stays in the interior of $D$, then the arc will consist
of more than one point. Hence the robber will not be located.
\end{proof}

From the point of view of practical applications it is also natural to
consider approximate version of the geometric localization game, where the Cop-player
is satisfied with determining the robber's position up to some small error.

\begin{figure}[t]
\center
\includegraphics[width=0.7\textwidth]{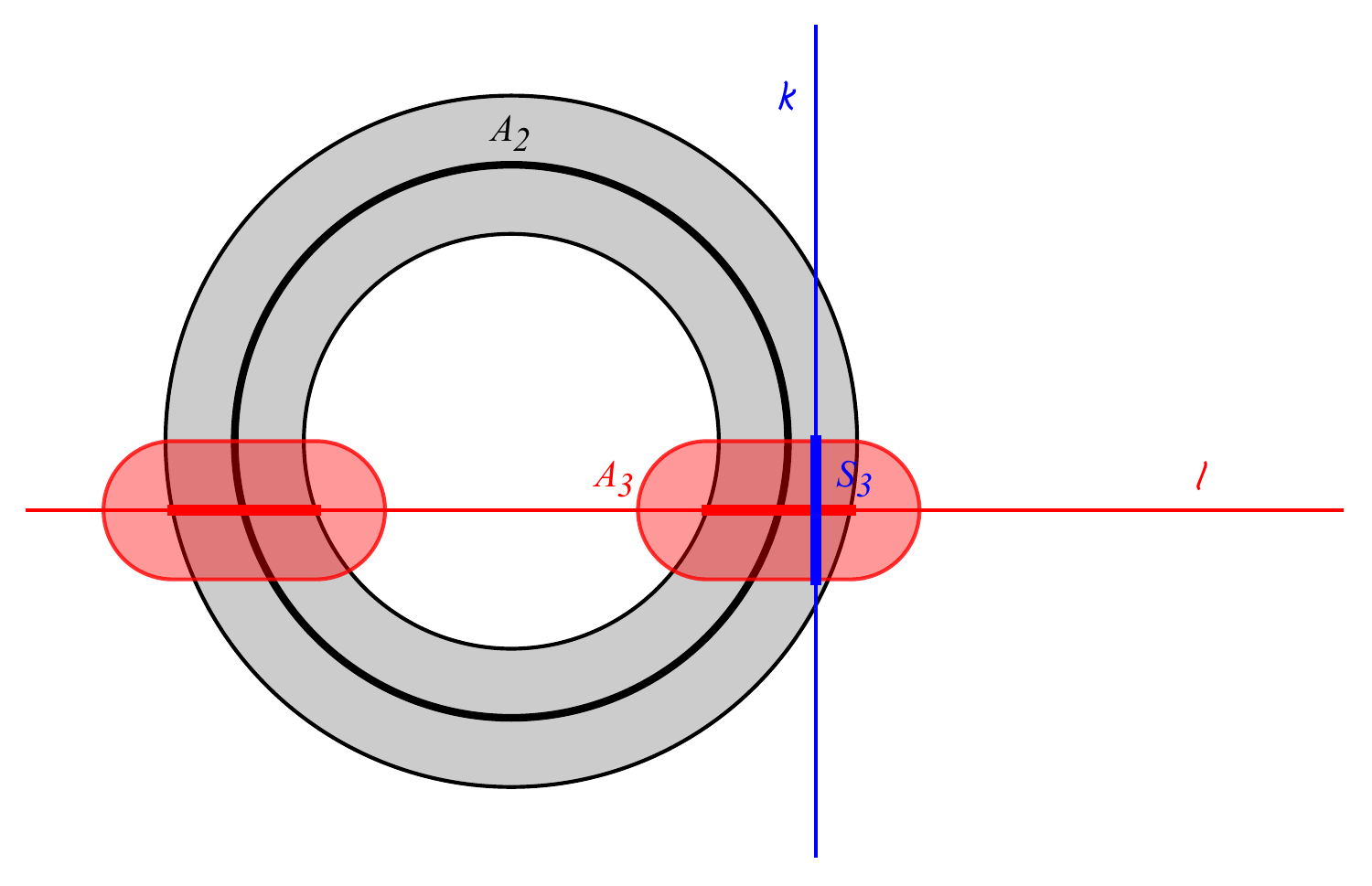}
\caption{Locating the robber on an arc close to a short segment.}\label{1eps}
\end{figure}

\begin{theorem}
For any $\varepsilon >0$, one cop can locate the robber with error of at
most $1+\varepsilon $. In other words one cop can determine a disk of radius $1+\varepsilon $ in which the robber is contained. 
\end{theorem}

\begin{proof}
Let $\varepsilon >0$ and $\delta$ such that $0<\sqrt{(1+\delta )^{2}+\delta ^{2}} \leq 1+\varepsilon$.
We will show that, in three rounds,
the Cop-player will know that the robber is on a single arc which is arbitrarily close to a
segment of length at most $2+2\delta $ (see Figure \ref{1eps}). Note that, in
the end of the $i$th round, the Cop-player knows that the robber is in the intersection of a
circle (with center in the checked vertex $c_{i}$ and radius equal to its
distance from the robber) and an area $A_{i}$, which she predicted to contain
the robber (with prediction being based on previous rounds). Hence, the robber is in the
set $S_{i}$ which is a collection of circular arcs. It is easy to see that 
\begin{equation*}
A_{i+1}=\{z\in \mathbb{R}^{2}:\mathrm{dist}(z,s)\leq 1\text{ for some }s\in
S_{i}\}\text{.}
\end{equation*}
The strategy for the Cop-player is to draw a line through two points of $A_{i}$ which
are at the largest possible distance from each other. Then in the next round
to choose a vertex $c_{i}$ from the line arbitrarily far from $A_{i}$. 

After first round, the Cop-player knows that the robber is on a circle and predicts that in
the next round he must be in an annulus $A_{2}$. In the second round, she
probes a vertex $c_{2}$ in the unbounded face of $\mathbb{R}^2 \setminus A_2$. The intersection $S_{2}$
of received circle and $A_{2}$ will consist of either one or two arcs.
Notice that the further the chosen vertex $c_{2}$ is from $A_{2}$, the
closer the arcs from $S_{2}$ are to be segments. In other words, we can find
a line $l$ such that for every point $x\in S_{2}$ there exists $y\in l$ with 
$\mathrm{dis}(x,y)<\delta $. Hence every point from 
\begin{equation*}
A_{3}=\{z\in \mathbb{R}^{2}:\mathrm{dist}(z,s)\leq 1\text{ with }s\in S_{2}\}
\end{equation*}
is at distance at most $1+\delta $ from $l$ (by the triangle inequality).%


Then in the third round, the Cop-player chooses a vertex $c_{3}$ to be 'far' from $A_{3}$
and on the line $l$. The intersection $S_{3}$ of $A_{3}$ and the resulting
circle is one arc. Again, if $c_{3}$ is far enough from $A_{3}$, then we can
find a line $k$ such that for every point $x\in S_{3}$ there exists $y\in k$
with $\mathrm{dis}(x,y)<\delta $. Moreover we can choose $k$ to be
perpendicular to the line $l$. Let $p$ be the point of the intersection of $%
k $ and $l$. Let $k^{\prime }$ be the intersection of $k$ and $A_{3}$. Then
the distance between $p$ and any point of $k^{\prime }$ is at most $1+\delta 
$. Hence the distance from $p$ to any point of $S_{3}$ (which contains
the robber's current location) is at most $\sqrt{(1+\delta )^{2}+\delta ^{2}}%
\leq 1+\varepsilon $. So, localizing the robber in $p$ is correct up to an error
of at most $1+\varepsilon $.
\end{proof}

\section{Further Work}\label{sec:conclusion}

In this paper, we have introduced the localization game (also introduced independently in~\cite{Haslegrave}). 
 Our study focused on upper bounds of $\zeta$ in planar graphs and computational complexity of this parameter in general graphs. Many questions remain open such that determining close formulas or bounds on $\zeta$ in other graph classes (e.g., outer-planar graphs, hypercube, partial cubes, chordal graphs, etc.). The question of the computational complexity of $\zeta$ in various graph classes such as bipartite graphs, split graphs (where computing the metric dimension is known to be NP-complete), bounded treewidth graphs, etc. is also of interest. Finally, most of the interesting turn-by-turn two-player games are known to be PSPACE-hard or even EXPTIME-complete. The exact status of the complexity of the localization game is still open.

It is also interesting to compare parameter $\zeta(G)$ to other graph invariants. For instance, it was proved in \cite{Johnson} that every graph with $\zeta(G)= 1$ is four colorable. This leads to the following conjecture.

\begin{conjecture}
	There is a function $f$ such that every graph with $\zeta(G)\leq k$ satisfies $\chi(G)\leq f(k)$.
\end{conjecture}
Another problem for future studies with more geometric flavor concerns unit disks graphs.
\begin{conjecture}
There is a function $f$ such that every unit disk graph satisfies $\zeta
(G)\leq f(\omega (G))$.
\end{conjecture}

\end{document}